\documentclass{amsart}

\usepackage{amsmath, amssymb, amsthm}
\usepackage{mathtools}
\usepackage{mathrsfs}
\usepackage{bbm}
\usepackage{cases}

\usepackage{graphicx} 
\usepackage{svg}
\usepackage{booktabs}
\usepackage{siunitx}
\usepackage{float} 
\usepackage{subfigure} % Note: Consider subfig or subcaption if conflicts occur

\usepackage{url}
\usepackage{enumitem}
\usepackage{color}
\usepackage{epstopdf}

\usepackage{hyperref} 

\usepackage{amsrefs}

\newtheorem{theorem}{Theorem}[section]
\newtheorem{lemma}[theorem]{Lemma}
\newtheorem{corollary}[theorem]{Corollary}
\newtheorem{proposition}[theorem]{Proposition}

\theoremstyle{definition}

\newtheorem{example}[theorem]{Example}

\theoremstyle{remark}
\newtheorem{remark}[theorem]{Remark}

\numberwithin{equation}{section}

\copyrightinfo{2026}{Jin Sun and Philipp S\"{u}rig}

\newcommand{\bracket}[1]{\left( #1\right)}

\newcommand{\av}[1]{\left\vert #1\right\vert}

\newcommand{\R}{\mathbb{R}}

\newcommand{\Lap}{\mathcal{L}}
\newcommand{\grad}{\nabla}

\newcommand{\diver}{\operatorname{div}}

\providecommand{\eat}[1]{}

\begin{document}

\title{Parabolic Frequency for Doubly Nonlinear Equations on Manifolds}
\author{Jin Sun}
    \address{Jin Sun, School of Mathematical Sciences, Fudan University, 200433, Shanghai, China}
    \email{jsun22@m.fudan.edu.cn}
\author{Philipp S\"{u}rig}
    \address{Philipp S\"{u}rig, Universit\"{a}t Bielefeld, Fakult\"{a}t f\"{u}r Mathematik, Postfach 100131, D-33501, Bielefeld, Germany}
    \email{philipp.suerig@uni-bielefeld.de}
\date{}

\begin{abstract}
We establish monotonicity formulas for a parabolic frequency function associated with sign-changing solutions to a class of doubly nonlinear parabolic equations of the form $\partial_t u = \Lap_{p,\varphi} u^q$ on weighted complete Riemannian manifolds without any curvature assumption, where $\Lap_{p,\varphi}$ denotes the weighted $p$-Laplacian and $p>1$, $q>0$. As a consequence, we obtain results on backward uniqueness for $q(p-1)\geq 1$ and unique continuation at infinity for $q(p-1) > 1$. We further consider equations with a controlled nonlinear perturbation term and derive an almost-monotonicity formula for the parabolic frequency. By employing the parabolic frequency, we also establish some Liouville-type results for ancient solutions in the case  $q(p-1)\geq 1$.
\end{abstract}

\let\thefootnote\relax\footnotemark\footnotetext{
The second author was funded by the Deutsche Forschungsgemeinschaft (DFG,
German Research Foundation) - Project-ID 317210226 - SFB 1283.}

\maketitle

\section{Introduction}

The frequency function method was introduced by Almgren \cite{Almgren1979} and systematically developed by Garofalo and Lin \cite{GarofaloLin1986} to study unique continuation for elliptic equations. Lin \cite{Lin1990} extended these techniques to parabolic equations, establishing a uniqueness theorem for solutions of the heat equation. Subsequently, Poon \cite{Poon1996} proved the monotonicity of parabolic frequency for the heat equation with bounded lower-order terms, from which strong unique continuation follows. Ni \cites{Ni2004,Ni2007} developed related entropy and monotonicity formulas on manifolds in connection with Li--Yau--Hamilton estimates. Li and Wang \cite{LiWang2019} obtained almost-monotonicity formulas on compact manifolds with curvature-dependent error terms. Recently, Colding and Minicozzi \cite{ColdingMinicozzi2022} proved parabolic frequency monotonicity on general Riemannian manifolds with the drift Laplacian without curvature assumptions. Baldauf and Kim \cite{BaldaufKim2022} extended frequency monotonicity to Ricci flows. Parabolic frequency is also considered in many other settings, see \cites{BanerjeeGarofalo2018, XSW2023, BHL2024, SunWang2025} for instance.

The doubly nonlinear parabolic equation $\partial_t u = \Delta_{p,\varphi} u^q$, in which both the diffusion operator and the nonlinearity in $u$ contribute to degeneracy or singularity, originates in Leibenson's \cite{Leibenson1945} modeling of turbulent gas filtration through porous media; see \cites{BenediktGirgKotrlaTakac2018,GrigoryanSurig2024} for historical accounts. Barenblatt \cite{Barenblatt1952} constructed the celebrated self-similar solutions on Euclidean spaces. The special case $p=2$ gives the porous medium equation, systematized in V\'{a}zquez \cite{Vazquez2007}, while $q=1$ yields the parabolic $p$-Laplace equation treated by DiBenedetto \cite{DiBenedetto1993}. The homogeneous case $\partial_t(|u|^{p-2}u)=\Delta_p u$, known as Trudinger's equation \cite{Trudinger1968}, has been extensively studied: H\"{o}lder regularity was obtained by Ivanov \cites{Ivanov1997,Ivanov1992} and Porzio--Vespri \cite{PorzioVespri1993}, Alt and Luckhaus \cite{AltLuckhaus1983} developed a general existence framework, and Harnack inequalities were established by Kinnunen--Kuusi \cite{KinnunenKuusi2007} and DiBenedetto--Gianazza--Vespri \cite{DiBenedettoGianazzaVespri2012}. A comprehensive variational approach was developed by B\"{o}gelein, Duzaar, Marcellini, and Scheven \cites{BDM2013, BDMS2018}. However, monotonicity formulas of frequency type for doubly nonlinear equations have not been established, and this is the main contribution of the present paper.

In parallel, Liouville-type theorems for ancient solutions are important in the
qualitative theory of parabolic equations. For the heat equation on complete
noncompact manifolds with nonnegative Ricci curvature, Souplet and Zhang
\cite{SoupletZhang2006} proved that bounded ancient solutions must be constant,
and Lin and Zhang \cite{LinZhang2019} classified ancient solutions of polynomial
growth as polynomials in time.  In the degenerate setting, DiBenedetto,
Gianazza, and Vespri \cites{DiBenedettoGianazzaVespri2010,
DiBenedettoGianazzaVespri2012} established via intrinsic Harnack inequalities
that bounded ancient solutions of the evolutionary $p$-Laplacian are constant.
For the doubly nonlinear equation, B\"{o}gelein, Duzaar, and Liao
\cite{BDL2021} and B\"{o}gelein, Duzaar, Liao, and Sch\"{a}tzler
\cite{BDLS2023} developed H\"{o}lder regularity that provides the foundational regularity framework and obtained a Liouville-type result for $q(p-1)\geq 1$. For bounded domains in Euclidean spaces, Hang and Lin \cite{HangLin1999} proved that for a broad class of elliptic equations, every nontrivial harmonic function has at least exponential growth, and this result was extended to the parabolic setting by Gui \cite{Gui2022}. However, Liouville-type
results for ancient solutions of doubly nonlinear equations on
weighted Riemannian manifolds have not been obtained; this is another
contribution of the present paper.

In this paper, we introduce a novel parabolic frequency function adapted to the doubly nonlinear equation on a weighted Riemannian manifold $(M,g,e^{-\varphi}\,dV_g)$ and establish its monotonicity.  Building on this monotonicity, we derive the strong unique continuation property at infinity for the slow diffusion case.  We further treat equations with a controlled nonlinear perturbation and prove almost-monotonicity of the parabolic frequency. As an application, we obtain a Liouville-type theorem for ancient solutions: in the case $q(p-1)>1$, the doubly nonlinear equation admits only trivial solutions, while in the critical case $q(p-1)=1$, it admits no solutions of polynomial growth. We note that all of our results extend to solutions of the doubly nonlinear equation on any relatively compact domain $\Omega\subset M$ with Dirichlet boundary condition.

Let $(M,g)$ be a complete Riemannian manifold and $\varphi:M\to\R$ be a smooth function. For $p>1$ and $q>0$, the weighted $p$-Laplacian (or drift $p$-Laplacian) is defined by
\[
\Lap_{p,\varphi} v \;=\; e^{\varphi}\,\diver\!\big(e^{-\varphi}\,|\nabla v|^{p-2}\nabla v\big).
\]
We consider sign-changing solutions $u:M\times[a,b]\to\R$ of 
\begin{equation}\label{eq:doublynonlinear}
	\partial_t u = \Lap_{p,\varphi} u^q,
\end{equation}
where $u^q:=|u|^{q-1}u$. 

The prototypical examples for equation~\eqref{eq:doublynonlinear} are the spherically symmetric self-similar solutions in $\R^n$ constructed by G.~I.~Barenblatt~\cite{Barenblatt1952}, now known as \emph{Barenblatt solutions}.

When $q(p-1) > 1$, the Barenblatt solution is given by
	\[
	u(x, t) = \frac{1}{t^{n/\beta}} \left( C - \varkappa \left( \frac{|x|}{t^{1/\beta}} \right)^{\frac{p}{p-1}} \right)_+^{\gamma},
	\]
	where $C > 0$ is any constant, and
	\[
	\beta = p + n[q(p-1) - 1], \quad \gamma = \frac{p-1}{q(p-1)-1}, \quad \varkappa = \frac{q(p-1)-1}{pq} \beta^{-\frac{1}{p-1}}.
	\]

When $q(p-1) = 1$, the Barenblatt solution takes the form
	\[
	u(x, t) = \frac{1}{t^{n/p}} \exp\left(-\zeta \left(\frac{|x|}{t^{1/p}}\right)^{\frac{p}{p-1}}\right),
	\]
	where $\zeta = (p-1)^2 p^{-\frac{p}{p-1}}$.

When $q(p-1) < 1$ and $\beta>0$, we have $\gamma<0$ and $\varkappa < 0$, and the Barenblatt solution becomes
	\[
	u(x, t) = \frac{1}{t^{n/\beta}} \left( C + |\varkappa| \left( \frac{|x|}{t^{1/\beta}} \right)^{\frac{p}{p-1}} \right)^{\gamma}.
	\]

In particular, when $q(p-1) \leq 1$, the solution $u(x,t)>0$ for all $x\in\R^n$ and $t>0$, exhibiting an infinite propagation speed. When $q(p-1) > 1$, the solution $u(x,t)$ is compactly supported for each $t>0$, exhibiting a finite propagation speed. Accordingly, the regime $q(p-1) \leq 1$ is referred to as the \emph{fast diffusion case}, while $q(p-1) > 1$ is the \emph{slow diffusion case}.

For a complete weighted Riemannian manifold $(M,g,e^{-\varphi}\,dV_g)$ of dimension $n$ and for $r>0$, the weighted Lebesgue space $L^{r}_\varphi(M)$ is defined by
\begin{align*}
L^r_\varphi(M) = \left\{ u: M \to \mathbb{R} \text{ measurable} \;\middle|\; \|u\|_{L^r_\varphi(M)} := \left( \int_M |u|^r \, e^{-\varphi} \, \mathrm{d}V_g \right)^{1/r} < \infty \right\},
\end{align*}
and the weighted Sobolev space $W^{1,r}_\varphi(M)$ is defined by  
\begin{align*}
W^{1,r}_\varphi(M) = \bigl\{ u \in L^r_\varphi(M) \;\big|\; \nabla u \text{ exists in the weak sense and } {\nabla u} \in \bracket{L^{r}_\varphi(M)}^n \bigr\},
\end{align*}
equipped with the norm
\begin{align*}
    \|u\|_{W^{1,r}_\varphi(M)} := \left( \int_M |u|^r \, e^{-\varphi} \, \mathrm{d}V_g +\int_M |\grad u|^r \, e^{-\varphi} \, \mathrm{d}V_g \right)^{1/r}.
\end{align*}
We denote by $W^{1,r}_{0;\varphi}(M)$ the closure of $C^{\infty}_c(M)$ in $W^{1,r}_\varphi(M)$.

Throughout the paper, we impose the following standing assumption to ensure that the parabolic frequency is well defined and that integration by parts is justified: for each $t\in[a,b]$ ($-\infty\leq a < b\leq +\infty$),
\begin{equation}\label{E:Assumption}
    u,u_t\in W^{1,q+1}_{0;\varphi}(M);\; \grad u^q\in \bracket{L^{p}_{\varphi}(M)}^n.
\end{equation}
This assumption is satisfied, for instance, when $ \int_M  \, e^{-\varphi} \, \mathrm{d}V_g<\infty$, the functions $u$, $u_t$, and $\grad u^q$ are bounded, and $ \lim_{R\to+\infty}\int_{\partial B_R}  \, e^{-\varphi} = 0$  for geodesic balls $B_R$ centered at a fixed point $o\in M$.

For a solution $u$ of \eqref{eq:doublynonlinear}, we define the weighted energies
\[
I(t)=\int_M u^{q+1}\,e^{-\varphi}\,dV_g,\qquad
D(t)=-\int_M |\nabla (u^q)|^p\,e^{-\varphi}\,dV_g,
\]
and the \emph{parabolic frequency}
\[
N(t)=\frac{D(t)}{I(t)}.
\]
Since $u^q=|u|^{q-1}u$, we have $u^{q+1}=\av{u}^{q+1}$, so $I(t)\geq 0$. Observe that $D(t)\le 0$ and therefore $N(t)\le 0$.  The choice of $I$ and $D$ is motivated by the natural energy structure of the equation: integration by parts gives
\[
\int_M u^q\,\Lap_{p,\varphi}u^q\,e^{-\varphi}\,dV_g = -\int_M |\nabla u^q|^p\, e^{-\varphi}\,dV_g = D(t).
\]
Set
\[
\delta=q(p-1)-1.
\]

Our first main result establishes the monotonicity of the parabolic frequency for the doubly nonlinear equation.

\begin{theorem}\label{thm:static}
Let $u$ satisfy $\partial_t u = \Lap_{p,\varphi}u^q$ and assumption~\eqref{E:Assumption}. Then
\begin{equation}\label{E:Derivative_of_N_t_1}
    N'(t)\ge \delta\,N(t)^2.
\end{equation}
In particular, if $\delta\ge 0$, then $N$ is monotone increasing. Moreover, if $\delta=0$, then $\log I(t)$ is convex. If $\delta\neq 0$, then $-\delta^{-1} I(t)^{-\delta/(q+1)}$ is convex.
\end{theorem}

As a direct consequence of Theorem~\ref{thm:static}, we obtain infinite extinction time for solutions when $\delta\geq 0$ and a lower bound on the extinction time when $\delta<0$.

\begin{corollary}\label{cor:backward}
Suppose $u$ solves \eqref{eq:doublynonlinear} on $M\times[a,b]$ with $u(\cdot,a)\not\equiv 0$ under assumption \eqref{E:Assumption}. If $\delta\ge 0$, then $u(\cdot, t)\not\equiv 0$ on $M$ for every $t\in [a,b]$. If $\delta < 0$, then $u(\cdot, t)\not\equiv 0$ on $M$ for every $t \in[a,b_0)$, where $b_0=\min\{1/(N(a)\delta)+a,b\}$.
\end{corollary}
In particular, Corollary~\ref{cor:backward} yields backward uniqueness when $\delta\geq 0$: if $u(\cdot,b)\equiv 0$, then $u$ must vanish identically on $M$ for all $t\in [a,b]$.

Beyond backward uniqueness, the monotonicity of $N$ in the slow diffusion case $\delta>0$ leads to the following strong unique continuation result at infinity. We first recall the definition of the vanishing order.

We say that a function $u:M\times(a,\infty)\to \R$ \emph{vanishes to order $k$ at $\infty$} if there exists a constant $C>0$ such that for all $t>a$,
\begin{equation}\label{E:I_k}
    I(t)\le C(t-a+1)^{-k}.
\end{equation}
Moreover, we say that a function $u:M\times(a,\infty)\to \R$ \emph{vanishes to infinite order at $\infty$} if for any integer $k>0$, there exists a constant $C>0$ such that \eqref{E:I_k} holds for all $t>a$.
\begin{corollary}\label{cor:unique_continuation}
    Let $u$ be a solution to \eqref{eq:doublynonlinear} on $M\times(a,\infty)$ satisfying assumption \eqref{E:Assumption},  with $\delta>0$. If $u$ vanishes to infinite order at $\infty$, then $u\equiv 0$.
\end{corollary}
It is worth noting that our definition of vanishing to infinite order at $\infty$ is weaker than that of Colding and Minicozzi \cite{ColdingMinicozzi2022}, who use the definition that  $u:M\times(a,\infty)\to \R$ vanishes to infinite order at $\infty$ if $\lim_{t\to \infty}e^{ct}I(t)=0$ for all constants $c$.

We next consider equations with a controlled nonlinear perturbation.  Specifically, we assume that $u$ satisfies
\begin{equation}\label{eq:error}
\big|\partial_t u - \Lap_{p,\varphi}u^q\big| \le 
\begin{cases}
    C(t)\bigl(|u| + |\nabla u^q|^{p/(q+1)}\bigr)&\quad \text{if}\ q\geq 1,\\
    C(t)\bigl(|u| + |u|^{1/2}|\nabla u^q|^{p/(2q+2)}\bigr)&\quad \text{if}\ 0<q< 1.
\end{cases}
\end{equation}
with $C(t)$ a non‑negative smooth function.  By carefully estimating the error terms arising from the perturbation via H\"{o}lder's and Young's inequalities, we obtain the following almost-monotonicity result.

\begin{theorem}\label{thm:error}
Suppose $u$ satisfies~\eqref{eq:error} on $M\times[a,b]$ under assumption~\eqref{E:Assumption},  with $\delta\ge 0$. Then
\begin{equation}
    \label{E:differ_I_error}
\frac{d}{dt}(\log I(t))\ge \bracket{q+1+C}N - \bracket{2q+3/2}C,
\end{equation}
and
\begin{equation}
    \label{E:differ_N_error}
\frac{d}{dt}N(t)\ge \frac{pq}{q+1}C^2\bracket{N-q-1/2}.
\end{equation}
\end{theorem}

This leads to the following backward uniqueness result.

\begin{corollary}
    \label{cor:error}
    Suppose $u$ satisfies~\eqref{eq:error} on $M\times[a,b]$ under assumption~\eqref{E:Assumption}, with $\delta\ge 0$ and $\int_a^b C(s)^2\,ds<\infty$. Then backward uniqueness holds: if $u(\cdot,b)=0$, then $u\equiv0$ for all $t\in[a,b]$.
\end{corollary}

A solution $u$ of \eqref{eq:doublynonlinear} is called an ancient solution if it is defined on $M\times (-\infty,T)$. Without loss of generality, we take $T=0$.

As an application of the monotonicity of the parabolic frequency, we establish the following Liouville-type theorem for ancient solutions when $\delta\geq 0$. 
\begin{theorem}\label{T:ancient_solution}
    Let $u$ be an ancient solution to \eqref{eq:doublynonlinear} satisfying assumption \eqref{E:Assumption}. Then the following hold.
	\begin{enumerate}
		\item[\textup{(i)}] If $\delta>0$ and  $0 < I(t) < \infty$ for all $t <0$, then $u$ is constant.
		\item[\textup{(ii)}] If  $\delta=0$ and $I(t)$ has at most polynomial growth, i.e., there exist $C, d > 0$ such that for all $t<0$, 
		\begin{equation*}
			I(t) \leq C\,\bracket{1+|t|}^d,
		\end{equation*}
		then $u$ is constant.
	\end{enumerate}
\end{theorem}
\begin{remark}
    {It is worth noting that none of the above theorems require the non-negativity of the solution $u$. We will always consider sign‑changing solutions. In particular,  $u^q$ is understood as $|u|^{q-1}u$,  $u^{q-1}$ as $|u|^{q-1}$ and $u^{q+1}$ as $|u|^{q+1}$.}
\end{remark}
All of the preceding theorems extend, with a slight adjustment in the proof (see Remark \ref{R:bounded_domains}), to solutions of the doubly nonlinear equation on any relatively compact domain $\Omega\subset M$ with Dirichlet boundary condition. More precisely, the solution $u$ satisfies
    \begin{equation}\label{E:bounded_domains}
    \begin{cases}
        \partial_t u = \Lap_{p,\varphi} u^q,\quad & (x,t)\in\Omega\times[a,b],\\
        u(x,t) = 0, & x\in\partial\Omega\times[a,b],
    \end{cases}
\end{equation}
and assumption \eqref{E:Assumption} is replaced by the assumption that for each $t\in[a,b]$ ($-\infty\leq a < b\leq +\infty$),
\begin{equation}\label{A:Assumption_bounded}
    u,u_t\in W^{1,q+1}_{0;\varphi}(\Omega);\; \grad u^q\in \bracket{L^{p}_\varphi(\Omega)}^n.
\end{equation}

The paper is structured as follows.  In Section~\ref{Sec2}, we prove Theorem~\ref{thm:static} together with Corollaries~\ref{cor:backward} and~\ref{cor:unique_continuation}, and compute the parabolic frequency of the Barenblatt solutions as an illustrative example.  Section~\ref{Sec3} is devoted to the almost-monotonicity property of the parabolic frequency for doubly nonlinear equations with lower-order terms, containing the proofs of Theorem~\ref{thm:error} and Corollary~\ref{cor:error}.  The Liouville-type results for ancient solutions are established in Section~\ref{sec4}.

\section{Parabolic frequency on manifolds}\label{Sec2}

In this section, in order to yield the convexity of $\log I(t)$ when $\delta=0$ and $-\delta^{-1}I(t)^{-\delta/(q+1)}$ when $\delta\neq 0$, we consider the following generalized parabolic frequency:
\begin{equation}\label{E:generalized_frequency}
    N_G(t) := \frac{D(t)}{I(t)^{\frac{pq}{q+1}}} = N(t)\cdot I(t)^{-\delta/(q+1)}.
\end{equation}
Then we have the following lemma.
\begin{lemma}\label{L:N_G>0}
    Let $u$ satisfy $\partial_t u = \Lap_{p,\varphi}u^q$ and assumption~\eqref{E:Assumption}. Then
\begin{equation}\label{E:Derivative_of_N_G}
    N_G'(t)\ge 0.
\end{equation}
\end{lemma}
\begin{proof}
    Set $v=u^q$ so that $v_t = q u^{q-1}u_t$ and $u_t = \Lap_{p,\varphi}v$. Differentiating $I$ and using the equation gives
\[
I'(t) = (q+1)\int_M u^q u_t e^{-\varphi} = (q+1)\int_M u^q \Lap_{p,\varphi}v \, e^{-\varphi}.
\]
Integrating by parts (noting that $\Lap_{p,\varphi}v = e^{\varphi}\diver(e^{-\varphi}|\nabla v|^{p-2}\nabla v)$) yields
\begin{equation}\label{E:integrate_I'}
    \int_M u^q \Lap_{p,\varphi}v \, e^{-\varphi}
= -\int_M \langle \nabla u^q, |\nabla v|^{p-2}\nabla v\rangle e^{-\varphi}
= -\int_M |\nabla v|^p e^{-\varphi} = D(t).
\end{equation}
Thus,
\begin{equation}\label{E:I'}
    I'(t)=(q+1)D(t).
\end{equation}
Now differentiate $D(t)$:
\[
D'(t) = -\int_M \partial_t\big(|\nabla v|^p\big) e^{-\varphi}
= -p\int_M |\nabla v|^{p-2}\langle\nabla v,\nabla v_t\rangle e^{-\varphi}.
\]
Integrating by parts gives
\begin{equation}\label{E:integrate_D'}
    \int_M |\nabla v|^{p-2}\langle\nabla v,\nabla v_t\rangle e^{-\varphi}
= -\int_M v_t \,\diver\!\big(|\nabla v|^{p-2}\nabla v e^{-\varphi}\big)
= -\int_M v_t e^{-\varphi}\Lap_{p,\varphi}v.
\end{equation}
Therefore,
\[
D'(t) = p\int_M v_t e^{-\varphi}\Lap_{p,\varphi}v
= pq\int_M u^{q-1} (\Lap_{p,\varphi}v)^2 e^{-\varphi}.
\]

We now turn to the computation of $N_G'$:
\begin{align*}
    N_G'I^{\frac{pq}{q+1}+1} 
    &=  D'I -\frac{pq}{q+1} DI'\\
    &= pq\bracket{\int_M u^{q-1}(\Lap_{p,\varphi} v)^2 e^{-\varphi}\;I - D^2}.
\end{align*}
Applying H\"{o}lder's inequality then yields,
\begin{align}\label{E:Holder_in_lemma_N_G}
    \Big(\int_M u^q\Lap_{p,\varphi}v\, e^{-\varphi}\Big)^2
&\le \Big(\int_M u^{q-1}(\Lap_{p,\varphi} v)^2 e^{-\varphi}\Big)
   \Big(\int_M u^{q+1} e^{-\varphi}\Big)\\
&= \Big(\int_M u^{q-1}(\Lap_{p,\varphi} v)^2 e^{-\varphi}\Big) I\notag.
\end{align}
Hence,
\[
\int_M u^{q-1}(\Lap_{p,\varphi} v)^2 e^{-\varphi}\, I \ge D^2.
\]
Thus, we obtain
\[
N_G'I^{\frac{pq}{q+1}+1} \ge 0,
\]
which gives the desired inequality. 
\end{proof}

Now we give the proof of the first main theorem.

\begin{proof}[Proof of Theorem \ref{thm:static}]
    First, from the definition \eqref{E:generalized_frequency}, we obtain:
    \begin{align*}
        N'(t)&= (N_G(t)I(t)^{\delta/(q+1)})'= N_G'(t)I(t)^{\delta/(q+1)} + \frac{\delta}{q+1} N_G(t)I(t)^{\delta/(q+1)-1}I'(t).
    \end{align*}
    It follows from \eqref{E:Derivative_of_N_G} and  the equality \eqref{E:I'} that
    \begin{align*}
        N'(t) \ge \delta N_G(t)I(t)^{\delta/(q+1)}\frac{D(t)}{I(t)}
        =\delta\, N(t)^2,
    \end{align*}
    which gives the inequality \eqref{E:Derivative_of_N_t_1}. 
    
    Note that for $\delta\neq 0$,
    \begin{align}\label{E:N_G_I_delta>0}
        N_G(t) = \frac{1}{q+1} \frac{I'(t)}{I(t)^{\delta/(q+1)+1}} = -\frac{1}{\delta} \bracket{I(t)^{-\delta/(q+1)}}'.
    \end{align}
    Therefore, Lemma \ref{L:N_G>0} implies that the function $-\delta^{-1}I(t)^{-\delta/(q+1)}$ is convex. 
    
    For $\delta=0$, $N(t) = N_G(t) = (\log I(t))'/(q+1)$. Therefore, $\log I(t)$ is convex.
\end{proof}
\begin{remark}\label{R:bounded_domains}
    Consider an open and relatively compact domain $\Omega\subset M$. Let $u$ be a solution to \eqref{E:bounded_domains} satisfying assumption \eqref{A:Assumption_bounded}. We define the weighted energies
    \[
        I_{\Omega}(t)=\int_{\Omega} u^{q+1}\,e^{-\varphi}\,dV_g,\qquad
        D_{\Omega}(t)=-\int_{\Omega} |\nabla (u^q)|^p\,e^{-\varphi}\,dV_g,
    \]
    and the \emph{parabolic frequency}
    \[
        N_{\Omega}(t)=\frac{D_{\Omega}(t)}{I_{\Omega}(t)}.
    \]
    Since $u=0$ and $u_t=0$ on $\partial\Omega$, the equalities \eqref{E:integrate_I'} and \eqref{E:integrate_D'}, where we use integration by parts, still hold for $u$. Therefore, a similar argument shows that Theorem \ref{thm:static}, as well as Theorem \ref{thm:error} and Theorem \ref{T:ancient_solution}, remain valid for solutions to the doubly nonlinear equation on any compact domain $\Omega \subset M$ with Dirichlet boundary condition.
\end{remark}
Using the monotonicity of parabolic frequency established in Theorem 1, we obtain the following corollary.
\begin{corollary}
    Let $u$ satisfy $\partial_t u = \Lap_{p,\varphi}u^q$ on $M\times[a,b]$ and assumption~\eqref{E:Assumption}. If  $\delta\geq 0$, then for any $t\in [a,b]$,
    \begin{equation}\label{E:cor_static_backward_delta>0}
        I(t)\ge I(a)\exp((q+1)N(a)(t-a)).
    \end{equation}
    If $\delta<0$ and $b_0=\min\{1/(N(a)\delta)+a,b\}$, then for any $t\in [a,b_0)$, 
    \begin{equation}\label{E:cor_static_backward_delta<0}
        I(t)\geq I(a)\bracket{\frac{1}{1-\delta(t-a)N(a)}}^{({q+1})/{\delta}}
    \end{equation}
\end{corollary}
\begin{proof}
    From Theorem \ref{thm:static}, if $\delta\ge0$, then $N(t)$ is monotone increasing. Thus,  for $t\ge a$,
    \begin{equation*}
        (\log I)' = (q+1)N \ge (q+1)N(a).
    \end{equation*}
    Integrating gives the inequality \eqref{E:cor_static_backward_delta>0}. 

    If $\delta<0$, then $\bracket{N(t)^{-1}}'\leq -\delta$, which means $N(t)^{-1}\leq N(a)^{-1}-\delta(t-a)$. Thus, for $t < 1/(N(a)\delta)+a$,
\[
N(t)\geq \frac{1}{N(a)^{-1}-\delta(t-a)}.
\]
Therefore, 
\[
\frac{d}{dt}(\log I(t))\geq \frac{q+1}{N(a)^{-1}-\delta(t-a)}.
\]
Integrating from $a$ to $t$ yields the inequality \eqref{E:cor_static_backward_delta<0}.
\end{proof}
As an immediate consequence, we obtain Corollary \ref{cor:backward}. 

\begin{remark}
Let $a=0$ and $b=+\infty$ and assume that $\delta<0$ and that $M$ admits the following Euclidean-type Sobolev inequality for $n>p$: 
\begin{equation}\label{Sobolevin}\left(\int_{M}{|v|^{\frac{pn}{n-p}} }\right)^{\frac{n-p}{n}}\leq C\int_{M}{|\nabla v|^{p}} \quad \textnormal{for all}~v\in W^{1, p}(M).\end{equation} 
Let $u_0$ be a nonnegative function. Since $q+1 = -\delta + pq$, we have
\begin{equation*}
    \left(\int_{M}{u_{0}^{q\frac{pn}{n-p}} }\right)^{-\frac{n-p}{n}}
    \leq \frac{\left(\int_{M}{u_{0}^{-\frac{\delta n}{p}} }\right)^{\frac{p}{n}}}{\int_{M}u_{0}^{-\delta+pq}}
    = \frac{\left(\int_{M}{u_{0}^{-\frac{\delta n}{p}} }\right)^{\frac{p}{n}}}{\int_{M}u_{0}^{q+1}}.
\end{equation*}
Then we can estimate the lower bound $T_{lower}=b_{0}$ of the extinction time further as 
\begin{align*}
T_{lower}=\frac{\int_{M}u_{0}^{q+1}}{-\delta\int_{M}|\nabla u_{0}^{q}|^{p}}\leq \frac{\int_{M}u_{0}^{q+1}}{-\delta C\left(\int_{M}{u_{0}^{q\frac{pn}{n-p}} }\right)^{\frac{n-p}{n}}}\leq C^{\prime} \left(\int_{M}{u_{0}^{-\frac{\delta n}{p}} }\right)^{\frac{p}{n}}.
\end{align*} 
If (\ref{Sobolevin}) holds, one of the authors proved in \cite{surig2024finite} the finite extinction time for solutions of $\partial_t u = \Lap_{p,0}u^q$. In particular, in the case $pq\leq \frac{-\delta(n-p)}{p}$, the extinction time obtained was $T=\left(\int_{M}u_{0}^{-(\delta n)/p}\right)^{p/n}$, so that indeed $T_{lower}\leq C^{\prime}T$.		
\end{remark}

In fact, Theorem \ref{thm:static} tells us that if $b=\infty$, then $u$ has a finite vanishing order at $\infty$, which depends on the constants $p $ and $ q$.
\begin{proposition}\label{P_vanishing_order}
    Let $u$ be a nontrivial solution to \eqref{eq:doublynonlinear} on $M\times(a,\infty)$ satisfying assumption \eqref{E:Assumption},  with $\delta>0$. Then the vanishing order of $u$ at $\infty$ is at most $(q+1)/\delta$.
\end{proposition}
 \begin{proof}
     Since $u$ is a nontrivial solution, it follows from Corollary \ref{cor:backward} that  $I(t)\neq 0$ for  all $t>a$. By Lemma \ref{L:N_G>0}, we have $0\geq N_G(t)\geq N_G(a+1)$ for all $t>a+1$.
     Using the equality \eqref{E:N_G_I_delta>0}, we obtain
     \begin{equation*}
          -\frac{1}{\delta} \bracket{I(t)^{-\delta/(q+1)}}'\geq N_G(a+1).
     \end{equation*}
     Integrating from $a+1$ to $t$ yields
     \begin{equation*}
          - \bracket{I(t)^{-\delta/(q+1)}-I(a+1)^{-\delta/(q+1)}}\geq \delta (t-a-1)N_G(a+1),
     \end{equation*}
     Thus, 
     \begin{equation*}
          I(t) \geq \bracket{I(a+1)^{-\delta/(q+1)}+\delta (t-a-1)(-N_G(a+1))}^{-\frac{q+1}{\delta}},
     \end{equation*}
     which implies that the vanishing order of $u$ at $\infty$ is at most $(q+1)/\delta$.
\end{proof}
 Therefore, Corollary \ref{cor:unique_continuation} directly follows from Proposition \ref{P_vanishing_order}, which gives the strong unique continuation at $\infty$ when $\delta>0$.
 ~\\

Since Barenblatt solutions are spherically symmetric self-similar solutions in $\R^n$, it is not hard to calculate their parabolic frequency.

\begin{example}[Barenblatt solutions on Euclidean spaces]
~\\
\hspace*{\parindent} 
    Here we compute the parabolic frequency of Barenblatt solutions on Euclidean spaces. 
    
    It is straightforward to see that for any $\delta\geq 0$, $I(t)$ and $D(t)$ are well-defined for every Barenblatt solution $u$. We now compute $I(t)$ for $\delta>0$.
    
    Using spherical coordinates \(r=|x|\) and denoting by \(\omega_n\) the area of the unit sphere in \(\mathbb{R}^n\), we have
    \[
    I(t)=\int_{\mathbb{R}^n} u^{q+1}\,dx = n\omega_n \int_0^\infty u(r,t)^{q+1} r^{n-1}\,dr.
    \]
    
    Substituting the expression for \(u\) and making the change of variable \(\xi = r/t^{1/\beta}\), we obtain
    \[
    u(r,t)^{q+1} = t^{-\frac{n(q+1)}{\beta}} \left( C - \varkappa \xi^{\frac{p}{p-1}} \right)_+^{\gamma(q+1)},\quad 
    r^{n-1}dr = t^{\frac{n}{\beta}} \xi^{n-1} d\xi.
    \]
    
    Hence
    \[
    \int_{\mathbb{R}^n} u^{q+1}\,dx = n\omega_n\, t^{-\frac{nq}{\beta}} 
    \int_0^{\xi_0} \left( C - \varkappa \xi^{\frac{p}{p-1}} \right)^{\gamma(q+1)} \xi^{n-1}\,d\xi,
    \]
    where \(\xi_0 = (C/\varkappa)^{\frac{p-1}{p}}\). Define
    \[
    A = \int_0^{\xi_0} \left( C - \varkappa \xi^{\frac{p}{p-1}} \right)^{\gamma(q+1)} \xi^{n-1}\,d\xi,
    \]
    which can be expressed in terms of the Beta function and is independent of $t$. Then
    \begin{align}\label{E:Barenblatt_solution_I_t}
    	I(t)=\int_{\mathbb{R}^n} u^{q+1}\,dx =  n\omega_n A\,t^{-\frac{nq}{\beta}}.
    \end{align}
    Therefore, 
    \begin{align}\label{E:Barenblatt_solution_N_t}
    	N(t) = \frac{I'(t)}{(q+1)I(t)} = -\frac{nq}{(q+1)(p+n\delta)}\cdot\frac{1}{t}.
    \end{align}
    Similarly, we can calculate that equations \eqref{E:Barenblatt_solution_I_t} and \eqref{E:Barenblatt_solution_N_t} also hold for $\delta=0$, where
    \begin{equation*}
	A = \int_0^{+\infty} \exp\left(-\frac{p\zeta}{p-1} \xi^{\frac{p}{p-1}}\right) \xi^{n-1}\,d\xi, \quad \zeta = (p-1)^2 p^{-\frac{p}{p-1}}.
    \end{equation*}

    In fact, for $-{p}/{n}<\delta<0$ ($\beta>0$), $I(t)$ and $D(t)$ are still well-defined for every Barenblatt solution $u$.
    
    Using spherical coordinates \(r=|x|\) and introducing the variable \(\xi = r/t^{1/\beta}\), we have
    \[
    I(t)=\int_{\mathbb{R}^n} u^{q+1}\,dx = n\omega_n \int_0^\infty u(r,t)^{q+1} r^{n-1}\,dr = n\omega_n \, t^{-\frac{nq}{\beta}} A,
    \]
where
    \[
    A = \int_0^\infty \left( C + |\varkappa| \xi^{\frac{p}{p-1}} \right)^{\gamma(q+1)} \xi^{n-1}\,d\xi.
    \]
    Since $\gamma = \frac{p-1}{q(p-1)-1}$, we have $A<+\infty$ if and only if
    \[
    \frac{p}{p-1}\gamma(q+1) + n = -\frac{p(q+1)}{1-q(p-1)} +n < 0,
    \]
    which is equivalent to $-\frac{p(q+1)}{n}<\delta<0$. Thus, $I(t)$ is well-defined.
    
    With a similar calculation, we have
    \[
    D(t) = C_1t^{-\frac{p(nq+1)-n}{\beta}}\int_0^\infty \left( C + |\varkappa| \xi^{\frac{p}{p-1}} \right)^{(q\gamma-1)p} \xi^{n-1+\frac{p}{p-1}}\,d\xi.
    \]
     Since for $\xi\rightarrow\infty$,
    \[
    \left( C + |\varkappa| \xi^{\frac{p}{p-1}} \right)^{(q\gamma-1)p} \xi^{n-1+\frac{p}{p-1}} \sim |\varkappa|^{(q\gamma-1)p} \xi^{\frac{p}{p-1}\bracket{(q\gamma-1)p+1}+n-1},
    \]
    which means $D(t)<+\infty$ if and only if 
    \[
    n + \frac{p}{p-1}\bracket{(q\gamma-1)p+1} = n + \frac{p(q+1)}{\delta} < 0.
    \]
    which means $D(t)$ is also well-defined for $-{p}/{n}<\delta<0$.
    
    Therefore, equalities \eqref{E:Barenblatt_solution_I_t} and \eqref{E:Barenblatt_solution_N_t}  hold for any $\delta>-{p}/{n}$, where $A$ is a constant depending on $n, p, q$ and $u$.
\end{example}
\begin{remark}
     A direct calculation shows that the inequality \eqref{E:Derivative_of_N_t_1} holds for any $\delta>-{p}/{n}$ and any Barenblatt solution. Moreover, for $\delta>0$, the equality \eqref{E:Barenblatt_solution_I_t} implies that the vanishing order of $u$ at $\infty$ is ${nq}/{\beta}$. Note that
     \begin{equation*}
         \frac{nq}{\beta} = \frac{nq}{p+n\delta}\leq \frac{q}{\delta}<\frac{q+1}{\delta},
     \end{equation*}
     which conforms to the conclusion of Proposition \ref{P_vanishing_order}.
\end{remark}

\section{More general operators}\label{Sec3}

We now consider the more general setting in which $u$ satisfies the differential inequality
\begin{equation}\label{E}
\big|\partial_t u - \Lap_{p,\varphi}u^q\big| \le 
\begin{cases}
    C(t)\bigl(|u| + |\nabla u^q|^{p/(q+1)}\bigr)&\quad \text{if}\ q\geq 1,\\
    C(t)\bigl(|u| + |u|^{1/2}|\nabla u^q|^{p/(2q+2)}\bigr)&\quad \text{if}\ 0<q< 1.
\end{cases}
\end{equation}
We show that an almost-monotonicity result still holds for the parabolic frequency.

\begin{proof}[Proof of Theorem \ref{thm:error}]
    Let $E = \partial_t u - \Lap_{p,\varphi}u^q$, so that by (\ref{E}) $|E|\le C(t)(|u| + |\nabla u^q|^{p/(q+1)})$ for $q\geq1$. Hence, 
\begin{align}
    I'&\nonumber=(q+1)\int_M u^qu_t e^{-\varphi}\\ 
    &\nonumber= (q+1)\int_M u^q\Lap_{p,\varphi}u^q e^{-\varphi} + (q+1)\int_M u^q E e^{-\varphi}\\
   &\nonumber = (q+1)D + (q+1)\int_M u^q E e^{-\varphi}\\
   &\label{Euse}\ge (q+1)D-(q+1)C\int_M \av{u}^{q}\bigl(|u| + |\nabla u^q|^{p/(q+1)}\bigr)e^{-\varphi}.
\end{align}
Using Young's inequality, we obtain 
\[
\av{u}^q|\nabla u^q|^{p/(q+1)} \le \frac{q}{q+1}u^{q+1} + \frac{1}{q+1} |\nabla u^q|^p,
\]
so that combining with (\ref{Euse}) gives
$${I^{\prime}\ge (q+1)D-C\int_M \bigl((2q+1)u^{q+1} + |\nabla u^q|^{p}\bigr)e^{-\varphi}}.$$
Thus, we obtain
\begin{align}\label{E:I'_q>1}
    \frac{d}{dt}(\log I(t))\ge \bracket{q+1+C}N - \bracket{2q+1}C.
\end{align}
Next, rewrite $D(t)$ and $I'(t)$ as
\begin{align*}
    D(t)&=\int_M u^q\,\bracket{u_t-\frac{1}{2}E}\,e^{-\varphi}-\frac{1}{2}\int_M u^q{E}\,e^{-\varphi},\\
    I'(t)&=(q+1)\bracket{\int_M u^q\,\bracket{u_t-\frac{1}{2}E}\,e^{-\varphi}+\frac{1}{2}\int_M u^q{E}\,e^{-\varphi}}.
\end{align*}
Hence, 
\[
I'(t)D(t)=(q+1)\left[\bracket{\int_M u^q\,\bracket{u_t-\frac{1}{2}E}\,e^{-\varphi}}^2-\frac{1}{4}\bracket{\int_M u^q{E}\,e^{-\varphi}}^2\right].
\]
For $q\geq 1$, differentiating $D(t)$ gives
\begin{align*}
    D'(t)&=p\int_M (u^q)_t e^{-\varphi}\Lap_{p,\varphi}u^q \notag\\
    &= pq\int_M u^{q-1}(u_t)^2 e^{-\varphi} - pq\int_M u^{q-1}u_t E e^{-\varphi}\notag\\
    &= pq\int_M u^{q-1}\bracket{\bracket{u_t-\frac{1}{2}E}^2-\frac{1}{4}E^2} e^{-\varphi}.
\end{align*}
Then, by H\"{o}lder's inequality, Young's inequality and $\delta\geq 0$, we obtain
\begin{align}
    N' I^2 \nonumber=& D'I - DI'\\
    \nonumber=&pq\int_M u^{q-1}\bracket{\bracket{u_t-\frac{1}{2}E}^2-\frac{1}{4}E^2} e^{-\varphi}\int_M u^{q+1}e^{-\varphi}\\
    \nonumber&-(q+1)\left[\bracket{\int_M u^q\,\bracket{u_t-\frac{1}{2}E}\,e^{-\varphi}}^2-\frac{1}{4}\bracket{\int_M u^q{E}\,e^{-\varphi}}^2\right]\\
    \nonumber\ge& (q(p-1)-1)\bracket{\int_M u^q\,\bracket{u_t-\frac{1}{2}E}\,e^{-\varphi}}^2-\frac{pq}{4}\int_M u^{q-1}{E^2} e^{-\varphi}\int_M u^{q+1}e^{-\varphi}\\
    \label{E:N'I^2_q>1}\ge&-\frac{pq}{4}\int_M u^{q-1}{E^2} e^{-\varphi}\int_M u^{q+1}e^{-\varphi}\\
    \nonumber\ge&-\frac{pq}{4}C^2(t)I(t)\int_M u^{q-1}\bracket{|u| + |\nabla u^q|^{p/(q+1)}}^2 e^{-\varphi}\\
    \nonumber\ge& -\frac{pq}{2}C^2(t)I(t)\bracket{I(t)+\frac{q-1}{q+1}I(t)-\frac{2}{q+1}D(t)}\\
    \nonumber=& -pqC^2(t)I(t)\bracket{\frac{q}{q+1}I(t)-\frac{1}{q+1}D(t)}.
\end{align}
Thus, we have
\begin{equation}\label{E:N'_q>1}
    N' \ge \frac{pq}{q+1}C^2\bracket{N-q}.
\end{equation}
For $0 < q < 1$, $|E|\le C(t)(|u| + |u|^{1/2}|\nabla u^q|^{p/(2q+2)})$ by (\ref{E}), and a similar calculation of \eqref{Euse} yields
\begin{align}\label{E:I'_q<1}
    \frac{d}{dt}(\log I(t))\ge \bracket{q+1+C/2}N - \bracket{2q+3/2}C.
\end{align}
By \eqref{E:N'I^2_q>1} and Young's inequality, we have
\begin{align}\label{E:N'_q<1}
    N' \ge \frac{pq}{q+1}C^2\bracket{N/2-q-1/2}.
\end{align}
Then the conclusion follows from combining \eqref{E:I'_q>1}, \eqref{E:N'_q>1}, \eqref{E:I'_q<1}, \eqref{E:N'_q<1} and $N\leq 0$.
\end{proof}

Applying Theorem \ref{thm:error}, we can show the following corollary. 

\begin{corollary}
    Let $u$ satisfy \eqref{eq:error} on $M\times[a,b]$ and assumption \eqref{E:Assumption},  with $\delta>0$. Then
\begin{align*}
    I(b) \ge& I(a)\exp\left((b-a)\bracket{q+1 + \sup_{[a,b]}C }\right.\\
    &\left.\times\left[\exp\left(\int_a^b \frac{pq}{q+1}C^2(s) \, ds\right)[N(a) - q - 1/2] -q-1\right]\right).
\end{align*}
\end{corollary}
\begin{proof}
    Since $N\leq 0$, we obtain from \eqref{E:differ_N_error}
\[
\frac{pq}{q+1}C^2\ge \frac{d}{dt}\log(q+1/2-N(t)).
\]
Integrating it from $a$ to $s$ gives
\[
\log(q+1/2-N(s))\leq\log(q+1/2-N(a))+\frac{pq}{q+1}\int_{a}^bC^2.
\]
Thus, for any $s\in[a,b]$,
\[
N(s)\ge \exp{\bracket{\frac{pq}{q+1}\int_{a}^bC^2}}\bracket{N(a)-q-1/2}+q+1/2.
\]
Therefore, from \eqref{E:differ_I_error}, we have
\begin{align*}
    \log I(b)-\log I(a)\ge& (b-a)\bracket{q+1 + \sup_{[a,b]}C }\\&\times\left[\exp\left(\int_a^b \frac{pq}{q+1}C^2(s) \, ds\right)[N(a) - q - 1/2] -q-1\right],
\end{align*}
which gives the desired inequality.
\end{proof}
Therefore, the backward uniqueness of Corollary \ref{cor:error} follows as a direct result.

\section{Liouville-type Theorem for Ancient Solutions}\label{sec4}

In this section, we give some applications of the parabolic frequency. Recall that for ancient solutions, assumption \eqref{E:Assumption} consists of the conditions that for each $t\in(-\infty,0)$,
\begin{equation*}
    u,u_t\in W^{1,q+1}_{0;\varphi}(M);\; \grad u^q\in \bracket{L^{p}_\varphi(M)}^n.
\end{equation*}

The first part of Theorem \ref{T:ancient_solution} is given by the following proposition, which shows that when $\delta > 0$ (the slow diffusion case), there is no nontrivial ancient solution to the doubly nonlinear equation \eqref{eq:doublynonlinear}.
\begin{proposition}\label{P:ancient_solution_delta>0}
    Let $u$ be an ancient solution to \eqref{eq:doublynonlinear} satisfying assumption \eqref{E:Assumption}, with $0 < I(t) < \infty$ for all $t <0$. Let $\delta > 0$. Then $u$ is constant.
\end{proposition}
\begin{proof}
    Suppose for contradiction that $N(t_0) < 0$ for some $t_0 < 0$.
    From Theorem \ref{thm:static}, we have $N'(t) \geq \delta\,N(t)^2$ with $\delta > 0$ and $N(t) \leq 0$. Define $\eta(t) = -1/N(t)$ wherever $N(t) < 0$.  Since $N(t) < 0$, we have $\eta(t) > 0$.
    When $N(t) < 0$, differentiation of  $\eta$ gives
\begin{equation}\label{E:ODE_eta}
    \eta'(t) = \frac{N'(t)}{N(t)^2}\ge \delta>0.
\end{equation}
Note that from Theorem \ref{thm:static}, $N(t)$ is monotone increasing, which means $N(t)\le N(t_0)<0$ for any $t < t_0$. Integrating \eqref{E:ODE_eta} backward  from $t_0$ to $t$ yields
\begin{equation*}
    \eta(t_0) -  \eta(t) \geq \delta(t_0 - t) .
\end{equation*}
Then we obtain that for any $t<t_0-\delta^{-1}\eta(t_0)$,
\begin{equation*}
    \eta(t) \leq \eta(t_0) - \delta(t_0 - t)<0,
\end{equation*}
which is a contradiction, since $\eta(t) = -1/N(t) > 0$ for any $t < t_0$. Therefore, $N(t) = 0$ for each $t<0$, which implies that
\begin{equation*}
    D(t)=-\int_M |\nabla (u^q)|^p\,e^{-\varphi}\,dV_g = 0.
\end{equation*}
Hence, $\nabla u^q = 0$ a.e. on $M$ for each $t$, so $u(\cdot, t)$ is constant a.e. for each $t$. Recall that
\begin{equation*}
    \partial_t u = \Lap_{p,\varphi} u^q = 0.
\end{equation*}
Thus, $u$ is constant a.e. on $M\times(-\infty,0)$. From the regularity result \cite{BDLS2023}, we know that $u$ is locally H\"{o}lder continuous, which means $u$ is constant on $M\times(-\infty,0)$.
\end{proof}
\begin{remark}
    In fact, integrating the ODE inequality \eqref{E:ODE_eta} forward gives that for any $t>t_0$,
\begin{equation*}
    \eta(t) \geq \eta(t_0) + \delta(t - t_0).
\end{equation*}
Since $\eta(t) = -1/N(t) > 0$ and $\eta(t_0) = -1/N(t_0) = 1/|N(t_0)| > 0$, this gives
\begin{equation*}
    |N(t)| = \frac{1}{\eta(t)} \leq \frac{1}{1/|N(t_0)| + \delta(t - t_0)}
\;\xrightarrow{t \to +\infty}\; 0.
\end{equation*}
Hence, for every solution to \eqref{eq:doublynonlinear} with $\delta > 0$, we have $N(t) \to 0^-$ as $t \to +\infty$. 
\end{remark}

For the case $\delta = 0$, we have the following Liouville-type result, which tells us that if for an ancient solution to \eqref{eq:doublynonlinear}, $I(t)$ has polynomial growth, then $u$ is constant.
\begin{proposition}\label{P:ancient_solution_delta=0}
    Let $u$ be an ancient solution to \eqref{eq:doublynonlinear} satisfying assumption \eqref{E:Assumption}. Let $\delta = 0$. Suppose $I(t)$ has at most polynomial growth, i.e., there exist $C, d > 0$ such that for all $t<0$, 
    \begin{equation}\label{E:assumption_polynomial}
        I(t) \leq C\,\bracket{1+|t|}^d.
    \end{equation}
    Then $N(t) \equiv 0$ for all $t < 0$, and $u$ is constant.
\end{proposition}
\begin{proof}
    Suppose for contradiction that $N(t_0) < 0$ for some $t_0 < 0$.
    
    With $\delta = 0$, inequality \eqref{E:Derivative_of_N_t_1} becomes:

$$
N'(t) \geq 0 \quad\text{for all } t < 0.
$$

Hence, $N(t)$ is a monotone increasing function of $t$, with values in $(-\infty, N(t_0)]$.
The limit
\begin{equation*}
    N_- \;:=\; \lim_{t \to -\infty} N(t) \;\in\; [-\infty, N(t_0)] 
\end{equation*}
exists (possibly $-\infty$). From \eqref{E:I'}, we obtain
\begin{equation*}
    \frac{d}{dt}\log I(t) = (q+1)N(t).
\end{equation*}
Integrating from $2t $ to $t_0$ for $t<t_0$ gives
\begin{equation*}
    \log I(t_0) - \log I(2t) = (q+1)\int_{2t}^{t_0} N(s)\,ds.
\end{equation*}
Since $N(t)$ is monotone increasing and $N(t)\leq 0$, we have
\begin{equation*}
    \log I(2t) - \log I(t_0) = (q+1)\int_{2t}^{t_0}\av{N(s)}\,ds\geq (q+1)\av{t}\av{N(t)}.
\end{equation*}
From the polynomial growth \eqref{E:assumption_polynomial}, we obtain that for any  $t<t_0$, 
\begin{equation*}
    \av{N(t)}\leq \frac{\log \bracket{C\,\bracket{1+|2t|}^d} - \log I(t_0)}{(q+1)\av{t}}.
\end{equation*}
Therefore, $N_-=\lim_{t \to -\infty} N(t)=0$, which is a contradiction, since $N_-\leq N(t_0) < 0$. Thus, $N(t) \equiv 0$ for all $t < 0$. With a similar argument as in Proposition \ref{P:ancient_solution_delta>0} and the regularity result \cite{BDL2021}, we deduce that $u$ is constant.
\end{proof}

Let $\Omega \subset \mathbb{R}^n$ be a bounded open set. Let $u$ be a solution of the heat equation $u_t = \Delta u$ on $\Omega \times (-\infty, 0]$ with initial data $u_0\in L^2(\Omega)$ and the Dirichlet boundary condition $u|_{\partial\Omega} = 0$. Let $\{\lambda_k, \varphi_k\}_{k=1}^{\infty}$ be the eigenvalues and eigenfunctions of $-\Delta$ on $\Omega$ with Dirichlet boundary conditions. Then the solution $u$ can be written as
\[
u(x,t) = \sum_{k=1}^\infty a_k e^{-\lambda_k t} \varphi_k(x), \qquad t \le 0,
\]
where $a_k = \int_\Omega u_0(x) \varphi_k(x) \, dx$. Thus,
\[
I_\Omega(t) = \int_\Omega |u(x,t)|^2 \, dx = \sum_{k=1}^\infty |a_k|^2 e^{-2\lambda_k t} = \sum_{k=1}^\infty |a_k|^2 e^{2\lambda_k |t|},
\]
which means $I_\Omega(t)$ has at least exponential growth. 
So, if $I_\Omega(t)$ satisfies the polynomial growth assumption \eqref{E:assumption_polynomial}, then $u$ vanishes in $\Omega\times(-\infty,0)$.

Therefore, Proposition \ref{P:ancient_solution_delta=0} generalizes the classic result of the heat equation to the doubly nonlinear setting, which is the following corollary.
\begin{corollary}
    Let $\Omega \subset M$ be an open, relatively compact set. Let $u$ be a solution to 
    \begin{equation*}
    \begin{cases}
        \partial_t u = \Lap_{p,\varphi} u^q,\quad & (x,t)\in\Omega\times(-\infty,0),\\
        u(x,t) = 0, & x\in\partial\Omega\times(-\infty,0).
    \end{cases}
\end{equation*}
Suppose $u$ satisfies assumption \eqref{A:Assumption_bounded} and the growth condition
\[
 I_{\Omega}(t) \leq C\bracket{1+|t|}^d \quad \text{for all } t < 0,
\]
where $C$ and $d$ are positive constants. Then $u\equiv 0$.
\end{corollary}

Finally, combining Proposition \ref{P:ancient_solution_delta>0} and Proposition \ref{P:ancient_solution_delta=0}, we obtain Theorem \ref{T:ancient_solution}.

\section*{Acknowledgments}
We thank Professor Bobo Hua for his helpful suggestions.

\bibliography{Parabolic_Frequency_for_Doubly_Nonlinear_Equations_on_Manifolds}
\end{document}